\documentclass[11pt,a4paper,oneside]{article}
\usepackage[english]{babel}
\usepackage[T1]{fontenc}
\usepackage{amsmath}
\usepackage{amsthm}
\usepackage{amsfonts}
\usepackage{amssymb}
\usepackage{lmodern} 
\usepackage{microtype} 			
\usepackage[numbers]{natbib}
\usepackage{hyphenat}
\usepackage[latin1]{inputenc}
\usepackage{indentfirst}
\renewcommand{\Re}{\operatorname{Re}}

\usepackage[plainpages=false,pdfpagelabels,pdfusetitle,pdfdisplaydoctitle, pdfstartpage=1, pdfstartview={{XYZ null null 1.00}}]{hyperref}
\hypersetup{
	pdftitle={},    
	pdfauthor={},    
	pdfnewwindow=true,      
	colorlinks=true,      
	citecolor=red}

\usepackage[top=2cm, bottom=2cm, left=2cm, right=2cm]{geometry}
\usepackage{mathtools}

\mathtoolsset{showonlyrefs}

\newcommand{\R}{\mathbb R}
\newcommand{\C}{\mathbb C}
\newtheorem{thm}{Theorem}[section]
\newtheorem{prop}[thm]{Proposition}
\newtheorem{lemma}[thm]{Lemma}
\newtheorem{rem}[thm]{Remark}

\newtheorem{coro}[thm]{Corollary}
\numberwithin{equation}{section}

\title{Threshold solutions for the $3d$ cubic INLS: \\ the energy-subcritical case}
\author{Luccas Campos, Luiz Gustavo Farah, and Jason Murphy} 
\date{}

\begin{document}
\maketitle

%
\begin{abstract}
\noindent	
We revisit the work [L. Campos and J. Murphy, \emph{SIAM J. Math. Anal.}, \textbf{55} (2023), pp. 3807--3843], which classified the dynamics of $H^1$ solutions at the ground state threshold for cubic inhomogeneous nonlinear Schr\"odinger equations of the form $i\partial_t u + \Delta u + |x|^{-b}|u|^2 u = 0$ in the range $b\in(0,\tfrac12)$. By modifying the modulation analysis and using Strichartz estimates in place of pointwise bounds, we extend the result to the full energy-subcritical range $b\in(0,1)$.  This strategy is expected to carry over to other dispersive equations with singular potentials.


\end{abstract}

\section{Introduction}

We consider the initial-value problem for inhomogeneous nonlinear Schr\"odinger (INLS) equations of the form
\begin{equation}\label{INLS}
\begin{cases} i\partial_t u + \Delta u + |x|^{-b}|u|^2 u = 0, \\
u(0,\cdot)=u_0(\cdot)\in H^1(\R^3),
\end{cases}
\end{equation}
where $u:\R_t\times\R_x^3\to\C$ and $b\in(0,1)$. 

For $H^1$-solutions, the INLS preserves the \emph{mass} and \emph{energy} defined by
$$
M[u(t)] := \int |u(t,x)|^2 \, dx\quad\text{and}\quad
E[u(t)] := \tfrac12 \int |\nabla u(t,x)|^2 \, dx 
- \tfrac14 \int |x|^{-b} |u(t,x)|^{4} \, dx,
$$
respectively.  The INLS equation is also invariant under a natural scaling transformation: if $u$ is a solution, then so is
$$
u_\lambda(t,x) = \lambda^{\frac{2-b}{2}} u(\lambda^2 t, \lambda x)
$$
for any $\lambda>0$. Under this scaling, the $\dot{H}^{s_c}(\mathbb{R}^3)$-norm with the {regularity}
$$
s_c = \tfrac{1+b}{2}
$$
remains invariant, that is, $\|u_\lambda(0)\|_{\dot{H}_x^{s_c}} = \|u(0)\|_{\dot{H}_x^{s_c}}$.  This notion of critical regularity for \eqref{INLS} plays a central role in the study of well-posedness, blow-up, and scattering of solutions.  The assumptions on $b$ make \eqref{INLS} an \emph{intercritical} equation.  Indeed, $b\in(0,1)$ corresponds to $s_c\in(\tfrac12,1)$. 

In this work, we are interested in so-called \emph{threshold} solutions satisfying
\begin{equation}\label{threshold}
M[u_0]^{1-s_c} E[u_0]^{s_c} = M[Q]^{1-s_c}E[Q]^{s_c},
\end{equation}
where $Q$ is the ground state, i.e. the unique nonnegative, radial $H^1$ solution to the elliptic equation
\begin{equation}\label{EqQ}
- Q + \Delta Q + |x|^{-b} |Q|^2 Q = 0.
\end{equation}
The ground state generates the standing wave solution $u(t,x) = e^{it} Q(x)$, a global, non-scattering solution to the INLS equation.

The classification of threshold behaviors for nonlinear dispersive PDE has received significant attention in recent years. In the context of the classical NLS, the pioneering work \cite{DM09} established the first results for the energy-critical problem. For related studies on this topic in wave and NLS-type equations, we refer the reader to \cite{DM08, LZ09, DR10, KMV21, M21, AI22, CFR22, YZZ22, DLR22, MMZ23, AM23, AMZ23}. 

The first analysis of threshold solutions for the INLS model was carried out in \cite{CM23}, where the first and third authors obtained the existence of two particular solutions to \eqref{INLS} (under the condition $b\in(0,\tfrac12)$):
\begin{thm}[Existence of particular solutions, \cite{CM23}]\label{thm:Qpm}
There exist forward-global radial solutions $Q^{\pm}$ to the equation \eqref{INLS} with
\[
M[Q^{\pm}] = M[Q] \quad \text{and} \quad E[Q^{\pm}] = E[Q],
\]
satisfying
\[
\|Q^{\pm}(t) - e^{it}Q\|_{H_x^{1}} \lesssim e^{-ct},
\]
for some $c>0$ and all $t>0$.

The solution $Q^{+}$ satisfies
\[
\|\nabla Q^{+}(0)\|_{L_x^{2}} > \|\nabla Q\|_{L_x^{2}},
\]
blows up in finite negative time, and satisfies $xQ^{+}\in L_x^{2}(\mathbb{R}^{3})$.

The solution $Q^{-}$ is global, satisfies
\[
\|\nabla Q^{-}(0)\|_{L_x^{2}} < \|\nabla Q\|_{L_x^{2}},
\]
and scatters in $H_x^{1}(\mathbb{R}^{3})$ as $t\to -\infty$.
\end{thm}

Moreover, they classified all possible threshold solutions.

\begin{thm}[Classification of threshold dynamics, \cite{CM23}]\label{thm:threshold}
Let $u_{0}\in H^{1}(\mathbb{R}^3)$ satisfying \eqref{threshold} and $u$ be the corresponding solution of \eqref{INLS}. We have the following

\begin{itemize}
\item[(i)] If
\[
\|u_{0}\|_{L_x^{2}}^{1-s_{c}} \|\nabla u_{0}\|_{L_x^{2}}^{s_{c}}
   < \|Q\|_{L_x^{2}}^{1-s_{c}} \|\nabla Q\|_{L_x^{2}}^{s_{c}},
\]
then $u$ either scatters as $t\to\pm\infty$ or $u=Q^{-}$ up to the symmetries of the equation.

\item[(ii)] If
\[
\|u_{0}\|_{L_x^{2}}^{1-s_{c}} \|\nabla u_{0}\|_{L_x^{2}}^{s_{c}}
   = \|Q\|_{L_x^{2}}^{1-s_{c}} \|\nabla Q\|_{L_x^{2}}^{s_{c}},
\]
then $u = e^{it}Q$ up to the symmetries of the equation.

\item[(iii)] If
\[
\|u_{0}\|_{L_x^{2}}^{1-s_{c}} \|\nabla u_{0}\|_{L_x^{2}}^{s_{c}}
   > \|Q\|_{L_x^{2}}^{1-s_{c}} \|\nabla Q\|_{L_x^{2}}^{s_{c}}
\]
and $u_{0}$ is radial or $xu_{0}\in L_x^{2}(\mathbb{R}^{3})$, then $u$ either blows up in finite positive and negative times or $u=Q^{+}$ up to the symmetries of the equation.
\end{itemize}
\end{thm}

We refer to solutions as in Theorem~\ref{thm:threshold}\textit{(i)} as \emph{constrained} solutions, while solutions as in Theorem~\ref{thm:threshold}\textit{(iii)} are called \emph{unconstrained}. 

In this work, we refine the analysis of \cite{CM23} in order to relax the restriction on $b$ to the full energy-subcritical range $b\in(0,1)$.

A key ingredient in the proof of the above results is the modulation analysis, which consists of studying solutions of \eqref{INLS} during the times that they remain close to the orbit of the ground state solution $e^{it}Q$. To this end, we introduce the functional
$$
\delta(t) = \left|\int |\nabla u(t,x)|^{2}\,dx - \int |\nabla Q(x)|^{2}\,dx\right|.
$$
By rescaling, we may consider $u_0 \in H^1$ satisfying $M[u_0]=M[Q]$ and $E[u_0]=E[Q]$.  We let $u:I\times \mathbb{R}^3\to \mathbb{C}$ be the corresponding threshold solution to \eqref{INLS}, and for $\delta_0>0$ we define
$$
I_0=\{t\in I: \delta(t)<\delta_0\}.
$$
If $\delta_{0}$ is chosen sufficiently small, the Implicit Function Theorem guarantees the existence of a function $\zeta:I_{0}\to\mathbb{R}$ such that
$$
g(t) = g_1(t)+ig_2(t)=e^{-i(\zeta(t)+t)}u(t)- Q, \quad \mbox{with} \quad (g_2(t), Q)=0.
$$
Here $(\cdot,\cdot)$ denotes the standard $L^2$ inner product.  By further decomposing $g(t)$ as
\begin{equation}\label{g2}
g(t)=\alpha(t)Q+h(t), \quad \mbox{where} \quad \alpha(t)=(g_1(t), \Delta Q)/(Q, \Delta Q),
\end{equation}
we obtain orthogonality conditions on $h(t)$ that play a key role in the analysis.

To proceed, we recall that important decay properties of the ground state solution were established in \cite[Proposition 2.2, Theorem 2.2, Lemma 2.7, and Lemma D.1]{GS08}. In particular, for $b \in (0,1)$, uniformly in $x\in \R^3$, we have
\begin{equation}\label{DecayQ}
\left|\nabla Q(x)\right| \lesssim \left|Q(x)\right| \sim (1+|x|)^{-1} e^{-|x|}.
\end{equation}

We then note that the restriction $b\in(0,\tfrac12)$ appearing in \cite{CM23} stems entirely from estimating the time derivative of the parameter $\alpha$.  From the equation \eqref{EqQ}, we derive that
\begin{equation}\label{DnQ}
|\nabla \Delta Q| \lesssim |x|^{-b-1}Q^3+|x|^{-b}Q^2|\nabla Q|+|\nabla Q|
\end{equation}
and thus \eqref{DecayQ}
yields that $\nabla \Delta Q \in L^2(\mathbb{R}^3)$ precisely in the range $b\in(0,\tfrac12)$.  Using this fact, the authors of \cite{CM23} were able to obtain suitable control over the parameter $\dot\alpha$ in this range.  (For further discussion, see Proposition \ref{ModThe} and Remark \ref{RCM23} below.)

In this note we present an alternative approach.  The key observation is that the pointwise bound on $|\dot\alpha|$ is ultimately integrated to obtain a uniform bound on $|\alpha(t)-\alpha(s)|$.  Thus it is only necessary to obtain an estimate on the time integral of $|\dot \alpha|$.  This, in turn, can be achieved by deriving suitable Strichartz-type estimates for the function $g$.  To prove such estimates, we make use of the equation satisfied by this function (see \eqref{solug} below).  This approach is less sensitive to the precise decay and smoothness properties of $Q$, and in particular it can be carried out to fruition in the full range $b\in(0,1)$.  We remark that this approach is in a similar spirit as the analysis of Dodson \cite{Dodson2024} on the characterization of blowup solutions for mass-critical NLS.

The rest of this paper is organized as follows. In Section \ref{MTR}, we review the modulation theory and establish a virial estimate. In Section~\ref{LpLqg} we establish $L_t^p L_x^q$ estimates for the function $g$. Finally, in Section \ref{Ualpha} we derive uniform estimates for $|\alpha(s) - \alpha(t)|$ in the full range $b\in(0,1)$.  Throughout the paper, we rely on standard notation and tools such as the $\lesssim$ notation and notation for space-time norms, as well as Strichartz estimates.  We refer the reader to \cite{CM23} for further details.

\section{Modulation Theory and Virial Estimate}\label{MTR}

Following the proof of \cite[Proposition 4.1]{CM23}, the decomposition \eqref{g2} satisfies the orthogonality conditions
\[
(h_{1}(t),\Delta Q)=0
\quad \text{and} \quad
(h_{2}(t),Q)=0,
\]
where $h=h_1+ih_2$.  These allow us to use the coercivity of the linearized operators (see \cite[Proposition 3.1]{CM23}) to deduce
$$
\|h(t)\|_{H_x^1}\sim |\alpha(t)|\sim \delta(t),
$$
and consequently
$$
\displaystyle \|g(t)\|_{H_x^1} \lesssim \delta(t).
$$
Next, recalling that $\zeta(t)$ denotes the phase parameter and observing that $g$ satisfies the equation
\begin{equation}\label{solug}
i\partial_t g + (\Delta-1) g  - \dot\zeta (g+Q) + R(g) = 0,
\end{equation}
where
\[
R(g)=|x|^{-b}(|g+Q|^2(g+Q)-|Q|^2Q) 
\]
we can further deduce
\begin{equation}\label{dottheta}
|\dot\zeta(t)| \lesssim \delta(t).
\end{equation}

As mentioned above, the control of $\dot{\alpha}$ in \cite[Proposition 4.1]{CM23} relied on the restriction $b\in(0,\tfrac12)$, which we aim to avoid in this paper.  Omitting this estimate, one has the following from \cite{CM23}.
\begin{prop}[Modulation theory]\label{ModThe}
Let $b\in(0,1)$ and $u:I\times \mathbb{R}^3\to \mathbb{C}$ be a $H^1$-solution to \eqref{INLS} with initial datum satisfying $M[u_0]=M[Q]$ and $E[u_0]=E[Q]$. For $\delta_0>0$ define
$$
I_0=\{t\in I: \delta(t)<\delta_0\}.
$$
Then, for small enough $\delta_0$, there exist functions $\zeta: I_0\to \mathbb{R}$, $\alpha: I_0\to \mathbb{R}$, $h: I_0\to H^1$ such that
$$
g(t)=e^{-i(\zeta(t)+t)}u(t)- Q =\alpha(t)Q+h(t), 
$$
with
\begin{equation}\label{alpha2}
\|h(t)\|_{H_x^1}\sim |\alpha(t)|\sim \delta(t), \quad  |\dot\zeta(t)| \lesssim \delta(t),
\end{equation}
and
\begin{equation}\label{nablag}
\|g(t)\|_{H_x^1} \lesssim \delta(t).
\end{equation}
\end{prop}

Furthermore, virial arguments can be employed to establish the control over $\delta(t)$ as follows.
\begin{prop}[Virial control]\label{VirialCon}
Let $b\in(0,1)$ and  $u:I\times \mathbb{R}^3\to \mathbb{C}$ be a forward-global $H^1$-solution to \eqref{INLS} with initial datum satisfying $M[u_0]=M[Q]$, $E[u_0]=E[Q]$ and $\|\nabla u_0\|_{L_x^2}\neq\|\nabla Q\|_{L_x^2}$. If $\|\nabla u_0\|_{L_x^2}>\|\nabla Q\|_{L_x^2}$, assume in addition that the solution is radial. Then, there exists $c>0$ such that
\begin{equation}\label{intdelta}
\int_{t}^{\infty}\delta(\tau)\,d\tau\lesssim e^{-ct} \quad \mbox{for all}\quad t> 0.
\end{equation}
\end{prop}
\begin{proof}
We first consider the case $\|\nabla u_0\|_{L_x^2} < \|\nabla Q\|_{L_x^2}$.  
By the variational characterization of $Q$, we have $\|\nabla u(t)\|_{L_x^2} < \|\nabla Q\|_{L_x^2}$ for all $t$, which in particular implies that $\delta(t)$ is bounded.  
From \cite[Lemma 5.3]{CM23} (independent of any control on $\dot{\alpha}$ and relying solely on the bound \eqref{nablag}), it follows that  
\begin{equation}\label{intdelta12}
\int_{t_1}^{t_2} \delta(\tau) \, d\tau \lesssim \delta(t_1) + \delta(t_2),  
\quad \text{for all} \quad t_2 \ge t_1 \ge 0.
\end{equation}
The boundedness of $\delta(t)$ then implies  
\[
\int_{0}^{\infty} \delta(\tau) \, d\tau < \infty.
\]
Therefore there exists a sequence $t_n \to \infty$ such that $\delta(t_n) \to 0$.  
Applying \eqref{intdelta12} yields  
\[
\int_{t}^{\infty} \delta(\tau) \, d\tau \lesssim \delta(t),  
\quad \text{for all} \quad t > 0.
\]
Gronwall's inequality then gives \eqref{intdelta} for such constrained forward-global solutions.

For unconstrained radial forward-global solutions, the argument is given in \cite[Proposition 6.1]{CM23} (see Steps 3 and 4).
\end{proof}

Our main goal is to establish a new estimate for $|\alpha(s)-\alpha(t)|$, which, in view of \eqref{alpha2} and \eqref{intdelta}, will imply that $\delta(t)\to 0$ as $t\to\infty$ for any forward-global threshold solution (see \cite[Corollary 4.3]{CM23} and Proposition \ref{final} below). 

\section{$L_t^pL_x^q$ control on $g$}\label{LpLqg}

%

We now analyze the solution $g$ of the equation \eqref{solug}.

\begin{lemma}[Strichartz bounds on $g$]\label{Strg}
Assume that the function $g$ given in Proposition \ref{ModThe} exists for sufficiently large times. If $I = [t,s]$ with $t\gg 1$, then there exists $c>0$ such that
\begin{equation}
    \| g\|_{S(L^2,I)}+\|\nabla g\|_{S(L^2,I)}  \lesssim _{\delta_0} \delta(t) + e^{-ct},
\end{equation}    
where $S(L^2,I)$ denotes the Strichartz space on $I\times\R^3$. 
\end{lemma}

\begin{proof}
Let $I=(t,s)$.  From Duhamel's principle and Strichartz estimates, we obtain
\begin{align}
\|g\|_{S(L^2,I)}+\|\nabla g\|_{S(L^2,I)} &\lesssim \|g(t)\|_{L_x^2}+\|\nabla g(t)\|_{L_x^2} + \|\dot\zeta  (g+Q)\|_{L^1_I L^2_x} +  \|\dot\zeta \nabla(g+Q)\|_{L^1_I L^2_x} \\
&\quad +\|R(g)\|_{L^1_I L^{2}_x} + \|\nabla R(g)\|_{L^2_I L^{6/5}_x}.
\end{align}
First observe that
\begin{align}
\|\dot\zeta \nabla(g+Q)\|_{L^1_I L^2_x}+\|\dot\zeta  (g+Q)\|_{L^1_I L^2_x}&\lesssim \|\dot\zeta\|_{L_I^1}\|Q\|_{H_x^1}+\|\dot\zeta\|_{L_I^1}(\|g\|_{L^{\infty}_I L^{2}_x}+\|\nabla g\|_{L^{\infty}_I L^{2}_x})\\
&\lesssim e^{-ct}+e^{-ct}(\|g\|_{S(L^2,I)}+\|\nabla g\|_{S(L^2,I)}),
\end{align}
where we have used \eqref{dottheta} and \eqref{intdelta} in the last inequality.

We next estimate the error term $R(g)$. From the inequality 
\begin{align}\label{Qineq}
Q|g|^2
\lesssim Q^2|g|+|g|^3,
\end{align}
we get
$$
\|R(g)\|_{L^1_I L^{2}_x}\lesssim \||x|^{-b}(Q^2|g|+|g|^3)\|_{L^1_I L^{2}_x}.
$$
We now note
\begin{align}
\||x|^{-b}Q^2g\|_{L^1_I L^2_x}&\lesssim \||x|^{-b}Q^2\|_{L_x^3}\|g\|_{L^1_I L^6_x}\lesssim \int_I \|\nabla g(\tau)\|_{L_x^2}\,d \tau \lesssim e^{-ct},
\end{align}
where we have used Sobolev embedding, \eqref{nablag}, \eqref{intdelta} and the fact that $\||x|^{-b}Q^2\|_{L_x^3}<\infty$ (cf. \eqref{DecayQ} and the fact that $b<1$).  We note that this estimate breaks down at the energy-critical case $b=1$\footnote{{The energy-critical case $b=1$ requires a different approach and will be treated in a companion work \cite{CFMarxiv25EC}.}}. Furthermore, we have by Hardy's inequality that
\begin{align}
\||x|^{-b}|g|^3\|_{L^1_I L^2_x}&\lesssim \||x|^{-b/3}g\|^3_{L^3_I L^6_x}\lesssim \||\nabla|^{b/3}g\|^3_{L^3_I L^6_x}.
\end{align}

Moreover, Sobolev embedding and interpolation yield
\begin{align}
\||\nabla|^{b/3}g\|^3_{L^3_I L^6_x}\lesssim \||\nabla|^{(b+1)/3}g\|^3_{L^3_I L^{18/5}_x}&\lesssim (\|g\|_{L^3_I L^{18/5}_x}+\|\nabla g\|_{L^3_I L^{18/5}_x})^3\\
&\lesssim (\|g\|_{S(L^2,I)}+\|\nabla g\|_{S(L^2,I)})^3,
\end{align}
since the pair $(3,18/5)$ is $L^2$-admissible.

Next, we estimate the term $\nabla R(g)$. Recall that
$$
|\nabla R(g)| \lesssim |x|^{-b}(Q^2|\nabla g|+Q|\nabla Q||g|+Q|g||\nabla g|+|g|^2|\nabla g|+|\nabla Q| |g|^2)+|x|^{-b-1}(Q^2|g|+Q|g|^2+|g|^3).
$$
Thus, from the inequalities \eqref{DecayQ}, \eqref{Qineq}, and 
\begin{align}\label{Qineq2}
Q|g||\nabla g|
\lesssim Q^2|\nabla g|+|g|^2|\nabla g|
\end{align}
we have
$$
\|\nabla R(g)\|_{L^2_I L^{6/5}_x}\lesssim \||x|^{-b}(Q^2|\nabla g|+Q^2|g|+|g|^2|\nabla g|+|g|^3)+|x|^{-b-1}(Q^2|g|+|g|^3)\|_{L^2_I L^{6/5}_x}.
$$
Now, again using that $\||x|^{-b}Q^2\|_{L_x^3}<\infty$, we deduce
\begin{align}
\||x|^{-b}Q^2|\nabla g|\|_{L^2_I L^{6/5}_x}&\lesssim \||x|^{-b}Q^2\|_{L_x^3}\|\nabla g\|_{L^2_I L^{2}_x}\\
&\lesssim \left(\int_I \|\nabla g(\tau)\|^2_{L^{2}_x}\,d\tau\right)^{1/2}\lesssim \left(\int_I \delta(\tau)\,d\tau\right)^{1/2}\lesssim e^{-ct/2},
\end{align}
where we have used \eqref{nablag}, \eqref{intdelta} and the fact that $\delta(t)\leq 1$. Moreover, from Hardy's inequality
\begin{align}\label{E1}
\||x|^{-b-1}Q^2|g|\|_{L^2_I L^{6/5}_x}&\lesssim \||x|^{-b}Q^2\|_{L_x^3}\||x|^{-1} g\|_{L^2_I L^{2}_x}\\
&\lesssim \|\nabla g\|_{L^2_I L^{2}_x}\lesssim e^{-ct/2},
\end{align}
as in the previous estimate. Next, we use \eqref{DecayQ} to deduce
\begin{align}
\||x|^{-b}Q^2 g\|_{L^2_I L^{6/5}_x}&\lesssim \||x|^{-b}Q^2\|_{L_x^{3/2}}\| g\|_{L^2_I L^{6}_x}\lesssim \|\nabla g\|_{L^2_I L^{2}_x}\lesssim e^{-ct/2},
\end{align}
arguing as before and using Sobolev embedding and the fact that $\||x|^{-b}Q\|_{L_x^{3/2}}<\infty$ (from \eqref{DecayQ} and the constraint $b<1$). Now, using Holder inequality, Hardy's inequality, Sobolev embedding, and interpolation, we get
\begin{align}
\||x|^{-b}|g|^2\nabla g\|_{L^2_I L^{6/5}_x}&\lesssim \||x|^{-b/2}g\|^2_{L^4_I L^{6}_x}\|\nabla g\|_{L^{\infty}_I L^{2}_x}\\
&\lesssim \||\nabla|^{(1+b)/2}g\|^2_{L^4_I L^{3}_x}\|\nabla g\|_{L^{\infty}_I L^{2}_x}\\
&\lesssim \left(\|g\|_{L^4_I L^{3}_x}+\|\nabla g\|_{L^4_I L^{3}_x}\right)^2\|\nabla g\|_{L^{\infty}_I L^{2}_x}\\
&\lesssim \left(\|g\|_{S(L^2,I)}+\|\nabla g\|_{S(L^2,I)}\right)^3,
\end{align}
since the pair $(4,3)$ is $L^2$-admissible. Moreover, applying Hardy's inequality, we obtain
\begin{align}\label{E2}
\||x|^{-b-1}|g|^3\|_{L^2_I L^{6/5}_x}&\lesssim \||x|^{-b/2}g\|^2_{L^4_I L^{6}_x}\||x|^{-1}g\|_{L^{\infty}_I L^{2}_x}\\
&\lesssim \||x|^{-b/2}g\|^2_{L^4_I L^{6}_x}\|\nabla g\|_{L^{\infty}_I L^{2}_x}\\
&\lesssim \left(\|g\|_{S(L^2,I)}+\|\nabla g\|_{S(L^2,I)}\right)^3,
\end{align}
as in the previous estimate.
Similarly,
\begin{align}
\||x|^{-b}|g|^3\|_{L^2_I L^{6/5}_x}&\lesssim \||x|^{-b/2}g\|^2_{L^4_I L^{6}_x}\|g\|_{L^{\infty}_I L^{2}_x}\\
&\lesssim \||\nabla|^{(1+b)/2}g\|^2_{L^4_I L^{3}_x}\| g\|_{L^{\infty}_I L^{2}_x}\\
&\lesssim \left(\|g\|_{L^4_I L^{3}_x}+\|\nabla g\|_{L^4_I L^{3}_x}\right)^2\|g\|_{L^{\infty}_I L^{2}_x}\\
&\lesssim \left(\|g\|_{S(L^2,I)}+\|\nabla g\|_{S(L^2,I)}\right)^3.
\end{align}
We collect the previous estimates and \eqref{nablag} to finally deduce
\begin{align}\label{Boot}
\|g\|_{S(L^2,I)}+\|\nabla g\|_{S(L^2,I)} &\lesssim \delta(t)+ e^{-ct/2}+\left(\|g\|_{S(L^2,I)}+\|\nabla g\|_{S(L^2,I)}\right)^3.
\end{align}
The result now follows from a continuity argument, provided $\delta_0$ is chosen sufficiently small. 
\end{proof}

\begin{coro}\label{Lintg}
Let $I=[t,s]$ with $t\gg 1$ and $a>0$. Then there exists $c>0$ such that
\begin{equation}\label{intg}
\int_I\|\nabla g(\tau)\|^a_{L_x^2}\,d\tau \lesssim e^{-ct}.
\end{equation}
\end{coro}

\begin{proof}
If $a\geq 1$, the desired estimate follows from \eqref{nablag}, \eqref{intdelta}, and the fact that $\delta(t)\leq 1$.  Thus we assume $0<a<1$. Since $I\subset \displaystyle\bigcup_{n=0}^{\infty}[t+n,t+n+1]$, by using \eqref{nablag}, Holder's inequality and \eqref{intdelta} we get
\begin{align}
\int_I\|\nabla g(\tau)\|^a_{L_x^2}\,d\tau &\lesssim \int_I[\delta(\tau)]^a \,d\tau\\
&\lesssim \sum_{n=0}^{\infty}\int_{t+n}^{t+n+1}[\delta(\tau)]^a \,d\tau\\
&\lesssim \sum_{n=0}^{\infty}\left(\int_{t+n}^{\infty}\delta(\tau) \,d\tau\right)^a\lesssim \sum_{n=0}^{\infty}e^{-ca(t+n)}\lesssim e^{-cat},
\end{align}
which implies the desired estimate.\end{proof}

\section{Uniform control on $\alpha$}\label{Ualpha}

Using the control on the function $g$ obtained in the previous section, we can now derive a uniform estimate for $|\alpha(s) - \alpha(t)|$, where $\alpha$ is the parameter introduced in \eqref{g2}. 

\begin{lemma}\label{Lemalpha} For some $c>0$ and $s\geq t\geq 0$, we have
\begin{equation}\label{exp_control+alpha}
    |\alpha(s)-\alpha(t)| \lesssim_{\delta_0} e^{-ct}
\end{equation}
    
\end{lemma}
\begin{proof}
Recall that
$$
g=\alpha Q+h
$$
and
$$
i\partial_t g + (\Delta-1) g  - \dot\zeta (g+Q) + R(g) = 0.
$$
We already proved that
$$
|\alpha(t)|\sim \|h(t)\|_{H_x^1}\sim \alpha(t)\quad \mbox{and} \quad |\dot\zeta(t)|\lesssim \alpha(t)\leq 1.
$$
By the orthogonality condition
\begin{align}
\dot\alpha\|\nabla Q\|_{L^2} = -\Re(\partial_t g,\Delta Q)_{L^2},
\end{align}
we have
$$
|\dot\alpha|\sim c\left|(-(\Delta-1) g  + \dot\zeta g - R(g),\Delta Q)_{L^2}\right|.
$$
Therefore,	 for $I=[t,s]$
\begin{equation}\label{alpha}
\begin{aligned}|& \alpha(s)-\alpha(t)| \\
&\lesssim \int_I|\dot\alpha(\tau)|\,d\tau\\
&\lesssim \int_I\int|\nabla g \nabla \Delta Q|\,dx\,d\tau+\int_I\int| g  \Delta Q|\,dx\,d\tau+\int_I\int| \dot\zeta g \Delta Q|\,dx\,d\tau+\int_I\int|R(g)\Delta Q|\,dx\,d\tau
\end{aligned}
\end{equation}

In view of \eqref{DecayQ} and \eqref{DnQ}, we have that $\nabla \Delta Q  \in L^{p}(\mathbb{R}^3),$ if $b<1$ and $p<3/2$. We then have
{
\begin{align}\label{gradg}
\int_I \int |\nabla \Delta Q| |\nabla g|\,dx\,d\tau &\leq \| \nabla \Delta Q\|_{L^{\frac{4}{3}}_x} \int_I \|\nabla g(\tau)\|_{L^{4}_x} \, d \tau\\
&\lesssim\int_I \|\nabla g(\tau)\|_{L^2_x}^{\frac14}\|\nabla g(\tau)\|_{L^6_x}^{\frac34}\,d\tau \\
&\lesssim \left[\int_I \|\nabla g(\tau)\|_{L^2_x}^{\frac{2}{5}} \, d \tau\right]^{\frac{5}{8}} \left[ \|\nabla g\|_{L^2_tL^6_x(I)}\right]^{\frac34}\lesssim e^{-{c}t},
\end{align}
}
where we have used Lemma \ref{Strg} and Corollary \ref{Lintg} in the last line.

Moreover,
\begin{align}
\int_I\int| g  \Delta Q|\,dx\,dt &\leq \| \Delta Q\|_{L^{6/5}_x} \int_I \|g(\tau)\|_{L^{6}_x} \, d \tau\\
&\lesssim \| Q+|x|^{-b}Q^3\|_{L^{6/5}_x} \int_I \|\nabla g(\tau)\|_{L^2_x}\,d\tau \lesssim
e^{-{c}t},
\end{align}
and since $|\dot\zeta(t)|\leq 1$, we also have 
\begin{align}
\int_I\int| \dot\zeta g  \Delta Q| &\lesssim e^{-{c}t}.
\end{align}

To control the final integral on the right hand side of \eqref{alpha}, we recall that $R(g) \sim |x|^{-b}(Q^2|g|+Q|g|^2+|g|^3)$. Thus in view of \eqref{Qineq} we have
\begin{align}
\int_I\int|R(g)\Delta Q|\,dx\,d\tau\lesssim \int_I\int|x|^{-b}|Q^2g \Delta Q|\,dx\,d\tau+\int_I\int |x|^{-b}|g|^3 |\Delta Q|\,dx\,d\tau.
\end{align}
Using the equation for $Q$ and \eqref{DecayQ} we have
\begin{align}
\int_I\int|x|^{-b}|Q^2g \Delta Q|\,dx\,d\tau&\lesssim \int_I\int|x|^{-b}|Q^3g|\,dx\,d\tau+\int_I\int |x|^{-2b}|Q^5g|\,dx\,d\tau\\
&\lesssim \int_I\int |x|^{-2b}|Q g|\,dx\,d\tau\\
&\lesssim \int_I \|g(\tau)\|_{L_x^6}\||x|^{-2b}Q\|_{L_x^{6/5}}\,d\tau\\
&\lesssim \int_I \|\nabla g(\tau)\|_{L_x^2}\,d\tau\lesssim e^{-{c}t}.
\end{align}
Moreover, 

\begin{equation}\label{xbg3Q}
\begin{aligned}
\int_I\int|x|^{-b}|g|^3 |\Delta Q|\,dx\,d\tau
&\lesssim \int_I\int |x|^{-2b}|Q|^{1/2}  |g|^3\,dx\,d\tau\\
&\lesssim \int_I \||g|^2(\tau)\|_{L_x^{12}}\|g(\tau)\|_{L_x^6}\||x|^{-2b}|Q|^{1/2}\|_{L_x^{\frac{4}{3}}}\,d\tau\\
&\lesssim \int_I \|g(\tau)\|^2_{L_x^{24}}\|\nabla g(\tau)\|_{L_x^2}\,d\tau\\
&\lesssim \int_I \|\nabla g(\tau)\|^2_{L_x^{\frac{8}{3}}} \|\nabla g(\tau)\|_{L_x^2}\,d\tau\\
&\lesssim \|\nabla g\|^2_{L_t^{\frac{16}{3}}L_x^{\frac{8}{3}}}\left(\int_I \|\nabla g(\tau)\|^{\frac{8}{5}}_{L_x^2}\,d\tau\right)^{\frac{5}{8}}\lesssim e^{-{c}t},
\end{aligned}
\end{equation}
following the same arguments as in \eqref{gradg} and using the fact that $(\frac{16}{3},\frac{8}{3})$ is $L^2$-admissible.

Collecting the above estimates we conclude the proof.
\end{proof}

\begin{rem}\label{RCM23} We can compare the estimates \eqref{gradg} and \eqref{xbg3Q} with those previously obtained in \cite[Proposition 4.1]{CM23}  in order to control the parameter $\dot{\alpha}$.  Using the Cauchy--Schwarz inequality, the authors of \cite{CM23} obtained  
\[
\left| (\nabla g, \nabla \Delta Q) \right| \le \|\nabla g\|_{L_x^2} \, \|\nabla \Delta Q\|_{L_x^2}
\]
and  
\[
\left| (|g|^3, |x|^{-2b} Q)\right| \le \|g\|_{L_x^6}^3 \, \||x|^{-2b} Q\|_{L_x^2}.
\]
Note, however, that $\nabla \Delta Q \in L^2(\mathbb{R}^3)$ and $|x|^{-2b} Q \in L^2(\mathbb{R}^3)$ only when $b < \tfrac12$ and $b < \tfrac34$, respectively. 
\end{rem}

Finally, we state the following result, which is the analogue of \cite[Corollary 4.3]{CM23}, using uniform bounds for $|\alpha(s) - \alpha(t)|$ appearing in Lemma~\ref{Lemalpha} instead of an estimate on $\dot{\alpha}$.

\begin{prop}\label{final}
Let $u:I\times \mathbb{R}^3\to \mathbb{C}$ be a forward-global $H^1$-solution to \eqref{INLS} with initial datum satisfying $M[u_0]=M[Q]$, $E[u_0]=E[Q]$ and $\|\nabla u_0\|_{L_x^2}\neq\|\nabla Q\|_{L_x^2}$. If $\|\nabla u_0\|_{L_x^2}>\|\nabla Q\|_{L_x^2}$, assume in addition that the solution is radial. Then there exist $c>0$ and $\zeta_0\in \mathbb{R}$ such that 
\begin{equation}\label{Limdelta}
\delta(t)\lesssim e^{-ct}\quad  \mbox{for all}\quad t>0
\end{equation}
and 
\begin{equation}\label{Limzeta_0}
\|u(t)-e^{i(\zeta_0+t)}Q\|_{H^1_x}\lesssim e^{-ct}\quad \mbox{for sufficiently large}\quad t.
\end{equation}
\end{prop}

\begin{proof}
From Proposition \ref{VirialCon}, estimate \eqref{intdelta} holds, and in particular, $\delta(t_n) \to 0$ along some sequence $t_n \to \infty$.  The relation \eqref{alpha2} then implies $\alpha(t_n) \to 0$, and Lemma \ref{Lemalpha} ensures that $|\alpha(t)|\lesssim e^{-ct}$, for all $t>0$.  
A further application of \eqref{alpha2} then gives \eqref{Limdelta}.

To prove \eqref{Limzeta_0}, note first that by \eqref{Limdelta} we have $t \in I_0$ for all sufficiently large times.  
Applying Proposition \ref{ModThe} then yields  
\begin{equation}\label{Limzeta_t}
\|u(t) - e^{i(\zeta(t) + t)} Q\|_{H^1_x} \lesssim \delta(t) \lesssim e^{-ct}
\end{equation}
and  
\[
|\dot{\zeta}(t)| \lesssim \delta(t).
\]
For sufficiently large $s \ge t > 0$, the Fundamental Theorem of Calculus combined with \eqref{intdelta} gives  
\[
|\zeta(s) - \zeta(t)| \lesssim \int_{t}^{s} \delta(\tau) \, d\tau \lesssim e^{-ct}.
\]
Hence, there exists $\zeta_0 \in \mathbb{R}$ such that $\displaystyle\lim_{t \to \infty} \zeta(t) = \zeta_0$.  
Finally, this limit together with \eqref{Limzeta_t} implies \eqref{Limzeta_0}.\end{proof}

With this result in hand, we can follow the analysis of \cite{CM23} and deduce the classification of all mass-energy threshold solutions as in Theorems \ref{thm:Qpm} and \ref{thm:threshold}, now under the relaxed condition $b\in(0,1)$.  This strategy may be adapted to study the classification of threshold dynamics in other dispersive models which involve singular potentials/nonlinearities.

\vspace{0.5cm}
\noindent 
\textbf{Acknowledgments.} L.C. was partially supported by Conselho Nacional de Desenvolvimento Cient\'ifico e Tecnol\'ogico - CNPq grants 307733/2023-8 and 404800/2024-6 and Funda\c{c}\~ao de Amparo a Pesquisa do Estado de Minas Gerais - FAPEMIG grants APQ-03186-24 and APQ-03752-25. J.M. was partially supported by NSF grant DMS-2350225, Simons Foundation grant MPS-TSM-00006622 and CAPES grant 88887.937783/2024-00. L.G.F. was partially supported by Conselho Nacional de Desenvolvimento Cient\'ifico e Tecnol\'ogico - CNPq grant 307323/2023-4. L.C. and L.G.F. were partially supported by Coordena\c{c}\~ao de Aperfei\c{c}oamento de Pessoal de N\'ivel Superior - CAPES grant 88881.974077/2024-01 and CAPES/COFECUB grant 88887.879175/2023-00. This work was completed while L.C. was a visiting scholar at University of Oregon in 2025-26 under the support of Conselho Nacional de Desenvolvimento Cient\'ifico e Tecnol\'ogico - CNPq, for which the author is very grateful as it boosted the energy into the research project. 

\addtocontents{toc}{\protect\vspace*{\baselineskip}}



\bibliographystyle{abbrvnat}

%
%
%

\newcommand{\Addresses}{{
  \bigskip
  \footnotesize

  L. Campos, \textsc{Department of Mathematics, Universidade Federal de Minas Gerais, Brazil}\par\nopagebreak
  \textit{E-mail address:} \texttt{luccas@ufmg.br}\\
  
L. G. Farah, \textsc{Department of Mathematics, Universidade Federal de Minas Gerais, Brazil}\par\nopagebreak
  \textit{E-mail address:} \texttt{farah@mat.ufmg.br}\\
  
J. Murphy, \textsc{Department of Mathematics, University of Oregon, Eugene, OR, USA}\par\nopagebreak
  \textit{E-mail address:} \texttt{jamu@uoregon.edu}  
}}
\setlength{\parskip}{0pt}
\Addresses
\batchmode
\end{document}